\documentclass[a4paper,10pt]{article}
\usepackage[utf8]{inputenc}
\usepackage{appendix}
\usepackage{amsthm,amssymb, amsmath,geometry,mathtools,breqn,array}
\usepackage{amsfonts}
\usepackage{blindtext,latexsym}
\usepackage[dvipsnames]{xcolor}
\newtheorem{theorem}{Theorem}[section]
\newtheorem*{theorem*}{Theorem}

\newtheorem{lemma}[theorem]{Lemma}
\newtheorem{proposition}[theorem]{Proposition}

 \newgeometry{vmargin={35mm,30mm}, hmargin={33mm,33mm}}
\title{On Signs of Hecke eigenvalues of Ikeda Lifts} 
\author{Nagarjuna Chary Addanki} 

\begin{document}
\maketitle

\section{Introduction} 

Siegel modular forms of genus $n$ and weight $k$ with respect to $\Gamma_n = Sp_{2n}(\mathbb{Z})$ are holomorphic functions on the Siegel upper half space $\mathbb{H}_n$. They are generalization of elliptic modular forms, the case $n=1$ gives us the classical modular forms over $SL_2(\mathbb{Z})$. Let $M_k(\Gamma_n)$ denote the space of Siegel modular forms and $S_k(\Gamma_n)$ denote the subspace of cuspidal forms. $S_k(\Gamma_n)$ has a special basis called Hecke eigenforms, these are eigenvectors with respect to operators on $S_k(\Gamma_n)$ called the Hecke operators. W. Duke and \"{O}. Imamo$\bar{g}$lu conjectured that given $f \in S_{2k-n}(\Gamma_1)$, there exists $F \in S_{2k}(\Gamma_n)$, such that the L-functions associated to $F$ and $f$ satisfy certain relations. For further details, check \cite{MR1707138}. Lifts from $n=1$ to $2$ are known as Saito-Kurokawa lifts. A construction of these lifts can be found in \cite{MR0532746}. Ikeda constructed lifts to higher genus. Theorem $3.2$ of \cite{MR1884618} explains the construction of $F_f \in S_{k+n}(\Gamma_{2n})$ from $f \in S_{2k}(\Gamma_1)$ for $n \equiv k(mod \ 2)$. These lifts are called Ikeda lifts and in this article we focus on Ikeda lifts for $n=4$ and higher. 

For each positive integer $m$ there is a Hecke operator associated to it, denoted by $T(m)$. For a Hecke eigenform $F$, let $\lambda_F(m)$ denote the eigenvalue of $T(m)$. In 2013, \cite{MR2726725} proved that for two normalized Hecke eigenforms $F,G \in S_k(\Gamma_1)$, if $\text{sign}(\lambda_F(p^r))= \text{sign}(\lambda_G(p^r))$ for almost all $p$ and $r$ then $F=G$. A similar result is proved in \cite{MR4198744} for $n=2$. Theorem $4$ of the article states that under a specific condition, if $F \in S_{k_1}(\Gamma_2)$ and $G \in S_{k_2}(\Gamma_2)$ are in orthogonal compliment of their respective Maass subspaces (the subspace of $S_k(\Gamma_2)$ generated by Saito-Kurokawa lifts) then half of the non zero coefficients in the sequence $\{ \lambda_F(m) \lambda_G(m) \}_{m \geq 1}$ are positive and half are negative. Saito-Kurokawa lifts can be characterized by the sign behavior of eigenvalues. Breulmann, in \cite{MR1719682}, showed that $F \in S_k(\Gamma_2)$ is a Saito-Kurokawa lift if and only if $\lambda_F(m) >0$ for all $ m \geq 1$. Kohnen, in \cite{MR2262899}, showed that a Hecke eigenform $ F \in S_k(\Gamma_2)$ is in the orthogonal compliment of Maass space  if and only if there are infinitely many sign changes in the sequence $\{ \lambda_F(m) \}_{  m \geq 1}$. Kohnen's and Breulmann's results can be combined to conclude that the space $S_k(\Gamma_2)$ breaks into two disjoint subspaces depending on the signs of eigenvalues of Hecke eigenforms. These are a few examples of incredible information stored in the sign behavior of the sequence $\{ \lambda_F(m) \}_{m \geq 1}.$

Mathematicians have studied this sign behavior through various questions. For example, Kohnen and Sengupta, in \cite{MR2326486}, proved that for $n=2$, the first negative eigenvalue occurs for some $m \ll k^2 log^{20}k$. Pitale and Schmidt, in \cite{MR2425722}, proved that there are infinitely many prime numbers $p$ such that the sequence of Hecke eigenvalues  $\{ \lambda_F(p^r) \}_{r \geq 1} $ has infinitely many sign changes. In this article, we focus on eigenvalues of Ikeda lifts of genus $4$ and higher. From Breulmann's result, $F \in S_k(\Gamma_2)$ is a Saito-Kurokawa lift if and only if $\lambda_F(m) \geq 0$ for all $m.$ We ask if a similar result is expected for Ikeda lifts. In this article we prove two theorems on the signs of eigenvalues of Ikeda lifts. The first result is true for Ikeda lifts of all genus. 
\begin{theorem}If $F_f \in S_{k+n}(\Gamma_{2n})$ is the Ikeda lift of $f \in S_{2k}(\Gamma_1)$ then for all large enough primes $p$, $\lambda_{F_f}(p) \geq 0.$
\end{theorem} 
The second result is for $n=4$. 
\begin{theorem} $F_f \in S_{k+2}(\Gamma_4)$ be the Ikeda lift of $f \in S_{2k}(\Gamma_1)$. Given $r$ there exists a positive constant $C_r$ such that  $\lambda_{F_f}(p^r) \geq 0$ for all $ p > C_r$. \end{theorem}

In \cite{MR0432552}, it is proved that for a general $n$, $-\lambda_F(p)$ is the coefficient of $p^{-s}$ in the spin L-function of $F$. Hence the spin L-function is enough to determine the sign of $\lambda_F(p)$. Schmidt, in \cite{MR1993950}, found an explicit expression for the spin L-function of the Ikeda lift $F_f$ in terms of the L-function of $f$. In section $2$, we talk about the basics of Siegel modular forms. We define objects like cuspidal forms, Hecke operators, spin L-functions, Petterson inner product and so on. Using Schmidt's formula for the spin L-function for Ikeda lifts and Andrianov's result, we compute $\lambda_F(p)$ and prove Theorem $1.1.$

To compute $\lambda_F(p^r)$ for higher powers of $p$, we need more information. Andrianov, in \cite{MR0432552}, proved that for any $F \in S_k(\Gamma_n)$, there are polynomials $Q_{F,p}(x)$ and $P_{F,p}(x)$ such that 
$$ \sum_{r=0}^\infty  \lambda_F(p^r)p^{-rs} = \frac{P_{F,p}(p^{-s})}{Q_{F,p}(p^{-s})} $$ and $Q_{F,p}(p^{-s})$ is the inverse of the spin L-function. Hence we need $P_{F,p}(x)$ and $Q_{F,p}(x)$ to compute $\lambda_F(p^r)$. $P_{F,p}(x)$ are known only up to $n= 4$, and the case of $n=4$ has been recently shown by Vankov in \cite{MR2780630}.  In section $3$ we compute $\lambda_F(p^r)$ for an Ikeda lift $F \in S_k(\Gamma_4)$. We use Vankov's result to compute the polynomial $P_{F,p}(x)$ and Schmidt's result to compute $Q_{F,p}(x)$. We use Andrianov's equation to find an explicit formula for $\lambda_{F}(p^r)$ and prove Theorem $1.2.$

\section{Siegel Modular Forms and Signs of $\lambda_F(p)$ for Ikeda lifts}

For any ring $R$, let $$G(R) \coloneqq GSp_{2n}(R)  =  \Big\{   g = \begin{pmatrix}
A & B\\
C & D 
\end{pmatrix} \in GL_{2n}(R) : g^t J g = \mu(g) J, \ J = \begin{pmatrix}
0 & 1_n\\
-1_n & 0 
\end{pmatrix}  \Big\} $$ where $\mu$ is the similitude homomorphism, $1_n$ is identity matrix of size $n$ and $A,B,C,D \in M_n(R)$. 
$$ Sp_{2n}(R) \coloneqq \{ g \in G(R) :  \mu(g) = 1 \} \quad \Gamma_n \coloneqq Sp_{2n}(\mathbb{Z}) .$$

A Siegel modular form F, of genus $n$ and weight $k$, with respect to $\Gamma_n$, is a holomorphic function on the Siegel upper half space
$$ \mathbb{H}_n \coloneqq \{ Z : Z \in M_n(\mathbb{C}), \ \prescript{t}{}{Z} =Z \ \text{and} \ Im(Z) >0 \} $$ satisfying
\begin{enumerate} \item modularity condition 
$$ F((AZ+B)(CZ+D)^{-1}) = det(CZ+D)^k F(Z) \quad \forall \ \begin{pmatrix}
A & B\\
C & D
\end{pmatrix} \in \Gamma_n \  \text{and} \  Z \in \mathbb{H}_n $$

\item and for $n=1$, $F(Z)$ is bounded on $\{ Z = X+i Y : Y \geq Y_0 \} \ \forall \ Y_0 > 0.$ \end{enumerate} 

Holomorphy and modularity imply that a Siegel modular form has a Fourier expansion of the form

$$ F(Z) = \sum_{\substack{ T= T^t \ T \geq 0 \\ T \ \text{half integral}}} A(T) e^{ 2 \pi i tr(TZ)} .$$

Let $M_k(\Gamma_n)$ denote the space of Siegel modular forms of genus $n$ and weight $k$. $F$ is called cuspidal if $A(T) = 0$ unless $T > 0$ and let $S_k(\Gamma_n)$ denote the subspace of cusp forms. For $F_1 \in S_k(\Gamma_n)$ and $F_2 \in M_k(\Gamma_n)$ we define the Petterson inner product 
$$ \langle F,G \rangle = \int_{\Gamma_n / \mathbb{H}_n} F_1(Z) \overline{F_2(Z)} \frac{dXdY}{det{Y}^{n+1-k}}.$$

For each $g \in G(\mathbb{Q}) \cap M_{2n}(\mathbb{Z})$ with a positive $\mu(g)$, a Hecke operator $T(g)$  on $M_k(\Gamma_n)$ is defined as follows,

$$ \text{for F}\in M_k(\Gamma_n), \ T(g)F \coloneqq \sum_i F|_kg_i \quad \text{where} \  \Gamma_n g \Gamma_n = \sqcup_i \Gamma_n g_i,  \quad g_i = \begin{pmatrix} 
    A_i & B_i  \\
    C_i & D_i  \\
  
    \end{pmatrix}$$
$$ \text{and} \ F|_kg_i(Z) = \mu (g)^{nk-\frac{n(n+1)}{2}}det(C_iZ+D_i)^{-k}F((AZ+B)(CZ+D)^{-1}) .$$
For every positive integer $m$, the above definition is extended to define $T(m)$ as follows 
$$ T(m)F \coloneqq \sum_{\mu(g)=m} T(g)F.$$ 

In Theorem $4.7$ of \cite{MR0432552}, it is shown that there exists a basis for $M_k(\Gamma_n)$ which are eigenforms with respect to all the Hecke operators simultaneously. Since $\displaystyle{\langle T(g)F,G \rangle =\langle F,T(g)G \rangle}$ we conclude that the eigenvalues must be real. For a Hecke eigenform $F \in S_k(\Gamma_n)$, denote $\lambda_F(g)$ as the eigenvalue for the operator $T(g)$. Classically $\lambda_F(g)$ can be expressed in terms of $Satake \ p-parameters$. For any $g$ with $\mu (g) = p^r$, depending on $F$ there are $n+1$ complex numbers $(a_{0,p},a_{1,p},\dots,a_{n,p})$ satisfying 

$$ \lambda(g) = a_{0,p}^r \sum_i \prod_{j=1}^n (a_{j,p}^{-j})^{d_{ij}} \ \text{where} \ \Gamma_n g \Gamma_n = \sqcup_i \Gamma_n g_i, $$
\[
g_i = \begin{pmatrix} 
    A_i & B_i  \\
    0 & D_i  \\
  \end{pmatrix}
\ \text{and} \ 
B_i = \begin{pmatrix} 
    p^{d_{i1}} &    & \ast \\
      & \ddots & \vdots \\
    0 & \dots  & p^{d_{in}} 
    \end{pmatrix}.
\]

The $a_{j,p}$ are called Satake  p-parameters of the eigenform $F$. See first $3$ chapters of \cite{MR3931351} for more on classical Siegel modular forms. 

Theorem $1.3.2$ of \cite{MR0432552} states that there are polynomials $P_{F,p}(x)$ and $Q_{F,p}(x)$ such that

\begin{equation}  \sum_{r = 0}^\infty \lambda_F(p^r)p^{-rs} =\frac{P_{F,p}(p^{-s})}{Q_{F,p}(p^{-s})} \end{equation}

and $$Q_{F,p}(p^{-s}) = (1-a_{0,p}p^{-s})\Big( \prod_{j=1}^n \prod_{1 \leq \delta_1 < \delta_2 <..< \delta_j \leq n} (1-a_{0,p}a_{{\delta_1},p}...a_{{\delta_j},p}p^{-s})\Big).$$

The denominator defines the spinor L-function of $F$, $L(s,F,spin)$ as follows  
$$ L(s,F,spin) \coloneqq \prod_p L_p(s,F,spin) \quad L_p(s,F,spin) = Q_{F,p}(p^{-s})^{-1}. $$

From equation $(1)$ it is clear that to compute $\lambda_F(p^r)$ both $P_{F,p}(x)$ and $L_p(s,F,spin)$ are necessary. If $f \in S_k(\Gamma_1)$ with Satake parameters $ap^{\frac{k-1}{2}}$ and $a^{-1}p^{\frac{k-1}{2}}$ then it is well known that $P_{f,p}(x) \equiv 1 $ and 
$$L_p(s,f,spin) =Q_{f,p}(p^{-s})^{-1} = \Big( (1-ap^{\frac{k-1}{2}}p^{-s})(1-a^{-1}p^{\frac{k-1}{2}}p^{-s}) \Big)^{-1}. $$

For genus 2, in \cite{MR0432552}, Andrianov showed that
$$ Q_{F,p}(p^{-s}) = 1-\lambda_F(p)p^{-s}+(\lambda_F(p)^2-\lambda_F(p^2)-p^{2k-4})p^{-2s}-\lambda_F(p)p^{2k-3}p^{-3s}+p^{4k-6}p^{-4s} $$ and $$ P_{F,p}(p^{-s}) = 1-p^{4k+2}p^{-2s}.$$

Vankov found the polynomials  $P_{F,p}(x)$ for genus $3$ and $4$ in \cite{MR2780630}. For higher genus they are still unknown to the best of our knowledge. 

Siegel modular forms are uniquely defined by their Fourier coefficients. Ikeda defines coefficients $A(B)$ for $F_f$ using the Fourier coefficients of $f$ and proves the following theorem in  \cite{MR1884618}.
\begin{theorem*}[Ikeda] Assume that $n \equiv k(mod \ 2)$. Then 
$$F_f(Z) = \sum_{B \in \mathcal{S}_{2n}(\mathbb{Z})^+} A(B)e^{2 \pi i tr(BZ)}, \quad Z \in \mathbb{H}_{2n} $$ is a Hecke eigen form in $S_{k+n}(\Gamma_{2n})$ and the standard L-function of $F$ is equal to 
$$ \zeta(s)\prod_{i=1}^{2n} L(s+k+n-i,spin,f).$$ 
Here $\mathcal{S}_{2n}(\mathbb{Z})^+$ is set of all positive definite half-integral matrices of size $2n$ and $\zeta(s)$ is Riemann zeta function. 

\end{theorem*}

We omit the details of construction of $A(B)$, check introduction of \cite{MR1884618} for more details. Let $F_f \in S_{k+n}(\Gamma_{2n})$ be the Ikeda lift of $f \in S_{2k}(\Gamma_1)$. To compute the spin L-functions of Ikeda lifts one has to move to the adelic language of modular forms. For each Hecke eigenform $F \in S_k(\Gamma_n)$ there is an associated automorphic representation $\Pi_F$ of $G(\mathbb{A})$, where $\mathbb{A}$ is ring of adeles of $\mathbb{Q}$. See Section $4$ of \cite{MR1821182} for details. Say $\Pi_{F_f}$, $\pi_f$ are automorphic representations corresponding to $F_f$ and $f$ respectively. Let $a,a^{-1}$ be the Satake-p-parameters for $\pi_f$. By Prop $2.2.1$ of \cite{MR1993950}, the spin L-function of $\Pi_{F_f}$ at a non archimedian place $p$ is given by 
\begin{equation} L_p(s, \Pi_{F_f}, spin ) = \prod_{j = 0}^n \prod_{\substack{r = j(j-2n) \\ Step2}}^{j(2n-j)} L_p( s+\frac{r}{2}, \pi_f , Sym^{n-j})^{\beta(r,j,n)} \end{equation} where 
$$ L_p(s, \pi_f , Sym^n) = \prod_{i=0}^n (1-a^i a^{-(n-i)}p^{-s})^{-1}, $$ $\beta(r,j,n) = \alpha(r,j,n)-\alpha(r,j-2,n)$ and $\alpha(r,j,n)$ is the number of ways to pick $j$ numbers from the set $\{1-2n,3-2n,...,-1,1,..., 2n-3,2n-1 \}$ such that they add up to $r$. The relation between $L(s,F_f,spin)$ and $L(s,\Pi_{F_f},spin)$ is explained by Lemma $10$ of \cite{MR1821182}. In the case of Ikeda lifts $F_f \in S_{k+n}(\Gamma_{2n})$, if 
$$L_p(s,\Pi_{F_f},spin)^{-1} = (1-a_0p^{-s}) \prod_{j=1}^n \prod_{1 \leq \delta_1 < \delta_2 <..< \delta_j \leq n} (1-a_0a_{\delta_1}...a_{\delta_j}p^{-s}) $$ 

then 
$$ L_p(s,F_f,spin)^{-1} = (1-p^{nk-\frac{n}{2}}a_0p^{-s}) \prod_{j=1}^n \prod_{1 \leq \delta_1 < \delta_2 <..< \delta_j \leq n} (1-p^{nk-\frac{n}{2}}a_0a_{\delta_1}...a_{\delta_j}p^{-s}).$$

\begin{theorem} For $n \equiv k(mod \ 2)$, let $F_f \in S_{k+n}(\Gamma_{2n})$ be the Ikeda lift of $f \in S_{2k}(\Gamma_1)$. For all large enough primes $p$, $\lambda_{F_f}(p) \geq 0$. \end{theorem} 
\begin{proof} 
By the statement of Theorem $1.3.2$ of \cite{MR0432552} the coefficient of $x$ in $Q_{F_f,p}(x)$ is $-\lambda_{F_f}(p)$. Using Schmidt's formula for L-function and Andrianov's result we can compute $\lambda_{F_f}(p)$ for an Ikeda lift.
In case of Ikeda lifts, $-p^{\frac{n}{2}-nk}\lambda_{F_f}(p)$ would be the coefficient of $p^{-s}$ in 
$$ L_p(s, \Pi_{F_f}, spin )^{-1} = \prod_{j = 0}^n \prod_{\substack{r = j(j-2n) \\ Step2}}^{j(2n-j)} L_p( s+\frac{r}{2}, \pi_f , Sym^{n-j})^{-\beta(r,j,n)}$$

For a fixed $j$, let us compute the coefficient of $p^{-s}$ in $\displaystyle{ L_p( s+\frac{r}{2}, \pi_f , Sym^{n-j})^{-1}}$ $$ L_p(s+\frac{r}{2},\pi_f,Sym^{n-j})^{-1} = \prod_{i=0}^{n-j}(1 - a^i (a^{-1})^{n-j-i} p ^{-(s+\frac{r}{2})})=\prod_{i=0}^{n-j}(1-a^{2i+j-n}p^{\frac{-r}{2}}p^{-s}).$$

Coefficient of $p^{-s}$ is 
$$ -\sum_{i=0}^{n-j} a^{2i+j-n}p^{-\frac{r}{2}}.$$

Hence the coefficient of $p^{-s}$ in $\displaystyle{ \prod_{j=0}^n \prod_{\substack{r = j(j-2n) \\ step 2}}^{j(2n-j)}L_p( s+\frac{r}{2}, \pi_f , Sym^{n-j})^{-\beta(r,j,n)}}$ is 
$$-\sum_{j=0}^n \Bigg( \sum_{i=0}^{n-j} a^{2i+j-n}\Bigg( \sum_{\substack{r = j(j-2n) \\ step2}}^{j(2n-j)} p^{\frac{-r}{2}} \beta(r,j,n) \Bigg) \Bigg)$$ 
\begin{equation} \Rightarrow \lambda_{F_f}(p) = p^{nk-\frac{n}{2}}\sum_{j=0}^n \Bigg( \sum_{i=0}^{n-j} a^{2i+j-n}\Bigg( \sum_{\substack{r = j(j-2n) \\ step2}}^{j(2n-j)} p^{\frac{-r}{2}} \beta(r,j,n) \Bigg) \Bigg) \end{equation} 

Since the above expression is a polynomial in $p^{\pm \frac{1}{2}}$, for a large enough $p$ we can determine the sign by calculating the leading term, i.e, the term with highest power of $p$. In the expression $(3)$, the highest power of $p$ appears when $r$ is the smallest. Every $r$ is of the form $j^2-2nj+b$ for some non negative integer $b$. Smallest value of $j^2-2nj+b$ is when $j=n$ and $b=0$. The leading term is $\beta(-n^2,n,n)p^{\frac{n^2}{2}}.$
$$ \beta(-n^2,n,n) = \alpha(-n^2,n,n)-\alpha(-n^2,n-2,n),$$

and the only way to pick $r$ numbers from $\{1-2n,3-2n,...,-1,1,..., 2n-3,2n-1 \}$ to add up to $-n^2$ is to pick all $n$ negative numbers and no positive numbers. We can conclude that $\alpha(-n^2,n,n) =1$, $\alpha(-n^2,n-2,n)=0$ and $\beta(-n^2,n,n)=1$.
$$ \Rightarrow \lambda_{F_f}(p) =  p^{\frac{n}{2}-nk} \big( p^{\frac{n^2}{2}} + \sum_{r=-n^2}^{n^2-1}c_r p^{\frac{r}{2}}\big) $$ 

For a fixed $r$ let us compute the coefficient $c_r$. From equation $(3)$ it is clear that $c_r = \sum_{j=0}^n c_{j,r}$, where $c_{j,r}$ is the coefficient of $p^{\frac{r}{2}}$ in 
$$ \sum_{i=0}^{n-j} a^{2i+j-n}\Bigg( \sum_{\substack{r = j(j-2n) \\ step2}}^{j(2n-j)} p^{\frac{-r}{2}} \beta(r,j,n) \Bigg)$$ 
$$ \Rightarrow c_{j,r} = \sum_{i=0}^{n-j} a^{2i+j-n} \beta(-r,j,n). $$
$$ \Rightarrow c_r = \sum_{j=0}^{n} c_{j,r} = \sum_{j=0}^{n} \bigg( \sum_{i=0}^{n-j} a^{2i+j-n} \beta(-r,j,n) \bigg) $$

Since $\lvert a \rvert \leq 1$ and $\beta(r,j,n)$ are bounded, we conclude that $c_r$ are bounded and the bound is independent of $p$. Hence for a large enough prime  
$$  p^{\frac{n^2}{2}} > \lvert \sum_{r=-n^2}^{n^2-1}c_r p^{\frac{r}{2}} \rvert . $$

This implies that $\lambda_{F_f}(p)$ is positive for all large enough primes $p$.

\end{proof}

\section{$\lambda_{F}(p^r)$ for Ikeda lifts of genus $4$}

In this section, using equation $(1)$ we compute $\lambda_{F_f}(p^r)$ for Ikeda lifts of genus $4$. Vankov calculated the polynomials $P_{F,p}(x)$ and $Q_{F,p}(x)$ as polynomials over the Hecke algebra. Hecke algebra, $\mathcal{H}_p$, over $G(\mathbb{Q})$ is the ring generated by the following Hecke operators
 
$$ T_i(p^2) = T(diag(\underbrace{1,\dots,1}_{n-i},\underbrace{p,\dots,p}_{i},\underbrace{p^2,\dots,p^2}_{n-i},\underbrace{p,\dots,p}_{i})) \ \text{for} \ 1 \leq i \leq n$$ 
$$\text{and} \  T(p) = T(diag(\underbrace{1,\dots,1}_n,\underbrace{p,\dots,p}_n)) .$$

Andrianov, in \cite{MR0241365}, proved that there exists polynomials $P_p(x)$ and $Q_p(x)$ in $\mathcal{H}_p[x]$ such that 
\begin{equation} \sum_{r = 0}^\infty T(p^r) x^r = \frac{P_p(x)}{Q_p(x)} . \end{equation}   

Vankov, in \cite{MR2780630}, calculated these polynomials explicitly for genus $4$. $P_{F_f,p}(x)$ and $Q_{F_f,p}(x)$ can be retrieved from $P_p(x)$ and $Q_p(x)$ by substituting the operators appearing in the coefficients with their respective eigenvalues of $F_f$. For example, in case of genus $1$. $$Q_p(x) = 1 - T(p)x + p^{k-1}x^2 \  \text{and} \ Q_{f,p}(x) = 1 - \lambda_f(p)x+p^{k-1}x^2 $$ 
$$ \Rightarrow L_p(s,f,spin) = (1-\lambda_f(p)p^{-s}+p^{k-1-2s})^{-1}.$$

Hence we compute spin L-function for genus $4$ Ikeda lift to obtain $Q_{F_f,p}(x)$, which will give us the eigenvalues for $F_f$ for the Hecke operators $T_1(p^2), \ T_2(p^2), \ T_3(p^2)$ and $T(p)$.

\begin{lemma} If $F_f \in S_{k+2}(\Gamma_4)$ is the Ikeda lift of $f \in S_{2k}(\Gamma_1)$ and the Satake parameters for $f$ are $ap^{\frac{k-1}{2}}$ and $a^{-1}p^{\frac{k-1}{2}}$ then
$$ Q_{F_f,p}(x) =  (1- a p^{2k-1}p^{\frac{1}{2}}x)(1-a^{-1}p^{2k-1} p^{\frac{1}{2}}x)(1- ap^{2k-1} p^{\frac{-1}{2}}x)(1-a^{-1}p^{2k-1} p^{\frac{-1}{2}}x)$$ $$(1- a p^{2k-1}p^{\frac{3}{2}}x)(1-a^{-1} p^{2k-1}p^{\frac{3}{2}}x)(1- a p^{2k-1}p^{\frac{-3}{2}}x)(1-a^{-1} p^{2k-1}p^{\frac{-3}{2}}x)(1- p^{2k-1}p^2x)(1- p^{2k-1}px)$$ 
$$(1- a^2 p^{2k-1}x)(1-a^{-2}p^{2k-1} x)(1-p^{2k-1} p^{-1}x)(1- p^{2k-1}p^{-2}x)(1-p^{2k-1}x)^2.$$ \end{lemma}

\begin{proof} Since $Q_{F_f,p}(p^{-s})^{-1} = L_p(s,F_f,spin)$, we use $(2)$ to calculate the spin L-function of $F$ and by extension $Q_{F_f,p}(p^{-s})$. In the case of Ikeda lift $F_f$ of genus $4$ we substitute $n=2$ in $(2)$. We get three sets of values for $\beta(r,j,2)$ depending on $j$. These can be found in appendix of \cite{MR1993950} or can be calculated using the definition of $\alpha(r,j,n)$ for small values of $n$. Let us compute $L_p(s,\pi_f,Sym^j)^{\beta(r,jn)}$ for each $j$.

\begin{enumerate} 

\item For $ j =0$, $r=0$ and $\beta(0,0,2) = 1.$ 

$$ L_p(s,\pi_f, Sym^2) = \Big(\ (1 - a^2 p^{-s})(1 - a^{-2} p^{-s})(1 - p^{-s}) \Big)\ ^{-1} $$

\item For $j=1$, $r= -3,-1,1,3$ and $ \beta(3,1,2)=\beta(-3,1,2) =\beta(1,1,2)= \beta(-1,1,2) = 1$ 
$$  L_p(s+\frac{-3}{2}, \pi_f, Sym^1)L_p(s+\frac{-1}{2}, \pi_f, Sym^1)L_p(s+\frac{1}{2}, \pi_f, Sym^1)L_p(s+\frac{3}{2}, \pi_f, Sym^1)  = $$
$$ \Big((1- a p^{\frac{3}{2} -s})(1-a^{-1} p^{\frac{3}{2}-s})(1- a p^{\frac{1}{2} -s})(1-a^{-1} p^{\frac{1}{2}-s})(1- a p^{\frac{-1}{2} -s})(1-a^{-1} p^{\frac{-1}{2}-s})$$ $$(1- a p^{\frac{-3}{2} -s})(1-a^{-1} p^{\frac{-3}{2}-s})\Big)^{-1}$$

\item For $j=2$, $r=-4,-2,0,2,4$ and $\beta(r,j,2) = 1$ for all $r$. 
$$ L_p(s-2,\pi_f, Sym^0)L_p(s-1,\pi_f, Sym^0)L_p(s,\pi_f, Sym^0)L_p(s+1,\pi_f, Sym^0)L_p(s+2,\pi_f, Sym^0)$$ 
$$= \Big((1- p^{-s+2})(1- p^{-s+1})(1- p^{-s})(1- p^{-s-1})(1- p^{-s-2})\Big)^{-1}$$
\end{enumerate} 
Combining all the cases we get $L_p(s,\Pi_{F_f},spin)$ as

$$\Big( (1 - a^2 p^{-s}) (1 - a^{-2} p^{-s})(1 - p^{-s})(1- a p^{\frac{3}{2} -s})(1-a^{-1} p^{\frac{3}{2}-s})$$
$$(1- a p^{\frac{1}{2} -s})(1-a^{-1} p^{\frac{1}{2}-s})(1- a p^{\frac{-1}{2} -s})(1-a^{-1} p^{\frac{-1}{2}-s})(1- a p^{\frac{-3}{2} -s})$$ 
$$(1-a^{-1} p^{\frac{-3}{2}-s})(1- p^{-s+2}) (1- p^{-s+1})(1- p^{-s})(1- p^{-s-1})(1- p^{-s-2})\Big)^{-1}$$

Replacing $p^{-s}$ in $L_p(s,\Pi_{F_f},spin)$ with $p^{2k-1} p^{-s}$ we get $L_p(s,F_f,spin)$ and hence the formula of $Q_{F_f,p}(x)$ as in the statement of the lemma. \end{proof} 

Coefficients of the polynomial $P_p(x)$ computed by Vankov in \cite{MR2780630} contain $T_1(p^2), \ T_2(p^2), \ T_3(p^2)$ and $T(p)$. To retrieve $P_{F_f,p}(x)$ from $P_p(x)$ we need eigenvalues of the aforementioned operators. Say $\lambda_{F_f}(T_i(p^2))$ is the eigenvalue for $F$ with respect to $T_i(p^2).$ 
\begin{lemma} The eigenvalues $\lambda_{F_f}(p)$ and $\lambda_{F_f}(T_i(p^2))$ for $i=1,2,3$ are 

$$ \lambda_{F_f}(T_3(p^2)) = \frac{\left(p^3+p^2+p+1\right) \left(a^2 p^{5/2}+a (p-1)+p^{5/2}\right) p^{4 (k-3)}}{a}$$ 
$$ \lambda_{F_f}(T_2(p^2)) = p^{4k-2}+ (a+\frac{1}{a})(p^{4k-\frac{7}{2}}+p^{4k-\frac{9}{2}}+p^{4k-\frac{11}{2}}-p^{4k-\frac{15}{2}}-p^{4k-\frac{17}{2}}-p^{4k-\frac{19}{2}})$$
$$ (a+\frac{1}{a})^2(p^{4k-3}+2p^{4k-4}+p^{4k-6}+p^{4k-7}) +p^{4k-4}(a^2+\frac{1}{a^2}+3) -p^{4k-8}-p^{4k-10}$$
$$ \lambda_{F_f}(T_1(p^2)) = p^{4k}+p^{4k-10}+p^{4k-4} + (a+\frac{1}{a})(p^{4k+\frac{1}{2}}+2p^{4k-\frac{1}{2}}-p^{4k-\frac{11}{2}}-p^{4k-\frac{13}{2}}-p^{4k-\frac{15}{2}})$$
$$+(a+\frac{1}{a})^2(p^{4k-1}+p^{4k-2}+p^{4k-3}-p^{4k-5}) -p^{4k-6}-p^{4k-7})(a^2+\frac{1}{a^2}+3) -p^{4k-8}-p^{4k-10}$$
$$+p^{4k-\frac{7}{2}}(a+\frac{1}{a})^3-p^{4k-6}(a^2+3+\frac{1}{a^2})+(p^{4k-\frac{5}{2}}+p^{4k-\frac{3}{2}})(a^3+4a+4\frac{1}{a}+\frac{1}{a^3})$$ 
$$ \lambda_{F_f}(p) = p^{2k-1}(a^2+\frac{1}{a^2}+\frac{a \left(p^3+p^2+p+1\right)}{p^{3/2}}+\frac{p^3+p^2+p+1}{a p^{3/2}}+p^2+\frac{1}{p^2}+p+\frac{1}{p}+2)$$ \end{lemma} 
\begin{proof} Since coefficients of $Q_p(x)$ also contain all the operators $T_i(p^2)$ and $T(p)$, we can compute the eigenvalues by comparing coefficient of $Q_{F_f,p}(x)$ and $Q_p(x).$ \end{proof} 

Above lemma gives an explicit expression for $\lambda_{F_f}(p)$. Using it we can improve the result of Theorem $2.1$ and prove that $\lambda_{F_f}(p)$ is positive \textbf{for all primes $p$} when $n=4.$

\begin{proposition}If $F_f \in S_{k+2}(\Gamma_4)$ is the Ikeda lift of $f \in S_{2k}(\Gamma_1)$ then $\lambda_{F_f}(p)$ is positive for all primes $p$. \end{proposition} 

\begin{proof} 
The formula of $\lambda_{F_f}(p)$ from Lemma $3.2$ can be rewritten as  
$$ p^{2k-1} \Big(p^{\frac{3}{2}}\Big( \sqrt{p}+a+\frac{1}{a} \Big)+p^{\frac{1}{2}}\Big( \sqrt{p}+a+\frac{1}{a} \Big)+p^{\frac{-3}{2}}\Big( \sqrt{p}+a+\frac{1}{a} \Big) +$$
$$2+p^{\frac{-1}{2}}\Big( a+\frac{1}{a} \Big)+ \frac{1}{p^2} + a^2+\frac{1}{a^2} \Big)$$

 Say $a+\frac{1}{a} = u$, then $\lambda_F(p)$ 
$$= p^{2k-1} \Big(p^{\frac{3}{2}}\Big( \sqrt{p}+u\Big)+p^{\frac{-1}{2}}u+p^{\frac{1}{2}}\Big( \sqrt{p}+u \Big)+p^{\frac{-3}{2}}\Big( \sqrt{p}+u \Big) + \frac{1}{p^2} + u^2 \Big)$$

Considering the expression as a quadratic in $u$, we can easily check when it is positive. 

\begin{enumerate} 

\item Assume $p \geq 5$. Since $ u \geq -2$, $\sqrt{p}+u \geq \sqrt{5}-2 > 0.236$ and $\frac{u}{\sqrt{p}} \geq -\frac{2}{\sqrt{5}} \geq -0.9.$ These inequalities imply that 
$$p^{2k-1}(p^{\frac{3}{2}}(\sqrt{p}+u)+\frac{u}{\sqrt{p}}) \geq p^{2k-1}( 5^{\frac{3}{2}}(0.236) -0.9) > 0$$ and 
$$  p^{2k-1}(p^{\frac{1}{2}}( \sqrt{p}+u)+p^{\frac{-3}{2}}( \sqrt{p}+u ) + \frac{1}{p^2} + u^2)  > 0.$$

Adding the above two inequalities, we conclude that $\lambda_{F_f}(p) \geq 0$ for all $p>5.$ 

\item If $p=3$ then 
$$\lambda_{F_f}(3) =3^{2k-1} (u^2+\frac{40}{3\sqrt{3}}u+\frac{112}{9}) $$ 

\item If $p=2$ then 
$$ \lambda_{F_f}(2)=2^{2k-1}(u^2+\frac{15}{2\sqrt{2}}u+\frac{27}{4}) $$

\end{enumerate}
In cases of $p=2,3$ we get quadratic equations which are positive for $u \geq -2$. This shows $\lambda_{F_f}(p)$ is positive for all primes.
\end{proof}

Substituting the eigenvalues found in Lemma $3.2$ into $P_p(x)$ computed by Vankov, we derive $P_{F_f,p}(x)$. And Lemma $3.1$ calculates $Q_{F_f,p}(x)$. With this information we can prove the following theorem. 

\begin{theorem} Let $F_f$ be a Hecke eigenform in $S_{k+2}(\Gamma_4)$ and the Ikeda lift of $f \in S_{2k}(\Gamma_1)$. For a fixed $r$, $\lambda_{F_f}(p^r)$ is positive for all large enough primes $p$. \end{theorem} 
\begin{proof}

To compute $\lambda_{F_f}(p^r)$ we mimic Bruelmann's idea(\cite{MR1719682}) of writing $1/Q_{F_f,p}(x)$ as partial fractions and comparing the powers of $x^r$ on both sides. 

$$ 
\frac{1}{Q_{F_f,p}(x)} = \frac{a_1}{(-1+p^{2k-1}x)^2} +\frac{a_2}{(-1+p^{2k-1}x)}+\frac{a_3}{(-a \sqrt{p}+p^{2k-1}x)}+\frac{a_4}{(-ap^{\frac{3}{2}}+p^{2k-1}x)} $$ 
$$ +\frac{a_5}{(-p^2+p^{2k-1}x)}+\frac{a_6}{(-\sqrt{p}+p^{2k-1}x)}+\frac{a_7}{(-1+a^2p^{2k-1}x)}+\frac{a_8}{(-a+\sqrt{p}p^{2k-1}x)}+$$ 
$$\frac{a_9}{(-1+a\sqrt{p}p^{2k-1}x)}+\frac{a_{10}}{(-1+p p^{2k-1}x)}+\frac{a_{11}}{(-a+p^{\frac{3}{2}}p^{2k-1}x)}+\frac{a_{12}}{(-1+ap^{\frac{3}{2}}p^{2k-1}x)}$$ 
$$+\frac{a_{13}}{(-1+p^2p^{2k-1}x)}+\frac{a_{14}}{(-a^2+p^{2k-1}x)}+\frac{a_{15}}{(-p+p^{2k-1}x)}+\frac{a_{16}}{(-p^{\frac{3}{2}}+ap^{2k-1}x)} $$
Precise values of $a_i's$ are given in appendix. 
Above fractions can be written as power series, for example 
$$ \frac{1}{-a+bx} = -\sum_{r=0}^\infty a^{-(r+1)}b^rx^r .$$

Writing every fraction as a power series and collecting like terms together, we get the coefficient of $x^r$ as  $$p^{(2k-1)r}(a_1(r+1)-a_2-a_3a^{-r-1}p^{-(r+1)/2}-a_4a^{-r-1}p^{-3(r+1)/2}-a_5p^{-2(r+1)}-a_6a^rp^{-(r+1)/2} $$ $$ -a_7a^{2r}-a_8a^{-r-1}p^{r/2}-a_9a^rp^{r/2}-a_{10}p^r-a_{11}a^{-r-1}p^{3r/2}-a_{12}a^rp^{3r/2}-a_{13}p^{2r}-a_{14}a^{-2r-2} $$ $$-a_{15}p^{-(r+1)}-a_{16}a^rp^{-3(r+1)/2}).$$

Let us denote it by $g(r)$. We can rewrite $(1)$ as 	 
$$ \sum_{r=0}^\infty \lambda_{F_f}(p^r) x^r = P_{F_f}(x) \sum_{r=0}^\infty g(r)x^r. $$

Comparing coefficients of $x^r$ on both sides, we have 
\begin{equation} \lambda_F(p^r) = \sum_{i=0}^{14} e_ig(r-i) \end{equation}

where it is understood that $g(r-i)=0$ for $r-i <0$. Here $e_i$ are defined by $\displaystyle{P_{F_f}(x) = \sum_{i=0}^{14} e_ix^i}$.

$\lambda_{F_f}(p^r)$ is expression of the form $f_1(p)r+f_2(p)$ where $f_i$ are rational functions in $p^{\frac{1}{2}}$ with coefficients in $\mathbb{Z}[a,a^{-1}]$. Hence for a fixed $r$, $\lambda_{F_f}(p^r)$ is a rational function in $p^{\frac{1}{2}}$ with coefficients in $\mathbb{Z}[a,a^{-1}]$. The denominator on right side of $(5)$ is 
$$ a^{2 r} p^{6 k+2 r+14}(a^2-1)^3(a^2+1)(p-1)^5(p+1)^3(p^2+1)(p^2+p+1)$$
$$ a^{17} (p+(a+\frac{1}{a})\sqrt{p}+1)(p-(a+\frac{1}{a})\sqrt{p}+1)^2(p-(a^3+\frac{1}{a^3})\sqrt{p}+1)(p^3-(a+\frac{1}{a})p^{\frac{3}{2}}+1)^3$$ 
$$(p^3+(a+\frac{1}{a})p^{\frac{3}{2}}+1)(p^4-(a^2+\frac{1}{a^2})p^2+1)(p^3-(a^3+\frac{1}{a^3})p^{\frac{3}{2}}+1)^3$$
$$ (p^3+(a+\frac{1}{a})p^{\frac{3}{2}}+1)(p^5-(a+\frac{1}{a})p^{\frac{5}{2}}+1)(p^7-(a+\frac{1}{a})p^{\frac{7}{2}}+1)$$
 
Expressions of the form $\displaystyle{p^m+(\mu+\frac{1}{\mu})p^{\frac{m}{2}}+1}$ are always positive for $p \geq 3$ and $\mu \in \mathbb{S}^1$. $(p-1)^5(p+1)^3(p^2+1)(p^2+p+1)p^{6k+2r+14}$ is also positive for all primes. Let us collect all these terms and call it $D(r)$, 
$$ D(r) = p^{6k+2r+14} (p-1)^5(p+1)^3(p^2+1)(p^2+p+1)(p+(a+\frac{1}{a})\sqrt{p}+1)(p-(a+\frac{1}{a})\sqrt{p}+1)^2$$ $$ (p-(a^3+\frac{1}{a^3})\sqrt{p}+1)(p^3-(a+\frac{1}{a})p^{\frac{3}{2}}+1)^3(p^3+(a+\frac{1}{a})p^{\frac{3}{2}}+1)(p^4-(a^2+\frac{1}{a^2})p^2+1)$$ $$(p^3-(a^3+\frac{1}{a^3})p^{\frac{3}{2}}+1)^3(p^3+(a+\frac{1}{a})p^{\frac{3}{2}}+1)(p^5-(a+\frac{1}{a})p^{\frac{5}{2}}+1)(p^7-(a+\frac{1}{a})p^{\frac{7}{2}}+1)$$  
Multiply $D(r)$ on both sides of equation $(5)$ to get 
\begin{equation} D(r) \lambda_F(p^r) =D(r) \sum_{i=0}^{14} e_ig(r-i) \end{equation}

Since $D(r)$ is positive for all prime $p \geq 3$, to determine the sign of $\lambda_{F_f}(p^r)$, it is enough to check the sign on the right side of the above expression. The numerator of $(6)$ is a polynomial in $p^{\frac{1}{2}}$ with coefficients in $\mathbb{Z}[a,a^{-1}]$ and leading term is
$$ p^{61+6k+3r+2kr}(a^{2r}(a^{25}-2a^{23}+2a^{19}-a^{17})).$$
  
The denominator of $(6)$ is $a^{2r+17}(a^2-1)^3(a^2+1)$. This implies that for a fixed $r$ the right side of equation $(6)$ is a polynomial in $p^{\frac{1}{2}}$ with the leading term $p^{61+6k+3r+2kr}$. 
$$\lambda_{F_f}(p^r) = p^{61+6k+3r+2kr}+\sum_{i=0}^{121+12k+6r+4kr} c_i p^{\frac{i}{2}} \ \text{s.t} \ c_i \in \mathbb{Z}[a,a^{-1}] $$
Since $c_i$ are bounded and the bound is dependent on $r$, for a fixed $r$
$$ \lambda_{F_f}(p^r) \geq  0  \ \text{for all large enough $p$.}$$
\end{proof}

\bibliography{Paper}
\bibliographystyle{alpha}

\newpage 
\begin{center}

{\huge \textbf{Appendix}} 

\hspace{1cm}

\begin{tabular}{ | m{1em} | m{14cm}|  } 
	
	\hline
	$a_1$ & $-a^6 p^7((a^2-1)^2 (p-1)^4 (p+1)^2 (a-\sqrt{p})^2 (a \sqrt{p}-1)^2 (a-p^{3/2})^2 (a p^{3/2}-1)^2)^{-1} $ \\ 
	\hline
	$a_2$ & $7 a^6 p^7((a^2-1)^2 (p-1)^4 (p+1)^2 (a-\sqrt{p})^2 (a \sqrt{p}-1)^2 (a-p^{3/2})^2 (a p^{3/2}-1)^2)^{-1} $ \\ 
	\hline
	$a_3$ & $a^2 p^{9/2}((a^2-1) (p-1)^3 (p+1) (a-\sqrt{p})^3 (a+\sqrt{p}) (a \sqrt{p}-1)^3 (a \sqrt{p}+1) (a^3 \sqrt{p}-1)  (a-p^{3/2}) (a p^{3/2}-1) (a^2 p^2-1) (a p^{5/2}-1))^{-1} $ \\ 
	\hline
	$a_4$ & $-a^2 p^2((a^2-1) (p-1)^3 (p+1) (p^2+p+1) (a-\sqrt{p}) (a \sqrt{p}-1)^3 (a \sqrt{p}+1) (a-p^{3/2}) (a p^{3/2}-1)^3 (a^2 p^2-1) (a p^{5/2}-1) (a^3 p^{5/2}+a^2 p^2+a^2 p+a p^{3/2}+a \sqrt{p}+1) (a p^{7/2}-1))^{-1} $ \\
	\hline
	$a_5$ & $-a^6 p^2((p-1)^5 (p+1)^3 (p^2+1) (p^2+p+1) (a-\sqrt{p}) (a \sqrt{p}-1) (a-p^{3/2}) (a p^{3/2}-1) (a^2-p^2) (a^2 p^2-1) (a-p^{5/2}) (a p^{5/2}-1) (a-p^{7/2}) (a p^{7/2}-1))^{-1} $ \\
	\hline 
	$a_6$ & $a^{17} p^{9/2}((a^2-1) (p-1)^3 (p+1) (a-\sqrt{p})^3 (a^3-\sqrt{p}) (a+\sqrt{p}) (a \sqrt{p}-1)^3 (a \sqrt{p}+1) (a-p^{3/2}) (a p^{3/2}-1) (a^2-p^2) (a-p^{5/2}))^{-1} $  \\
	\hline
	$a_7$ & $-a^{32} p^7((a^2-1)^3 (a^2+1) (a-\sqrt{p})^3 (a^3-\sqrt{p}) (a+\sqrt{p}) (a \sqrt{p}-1)^3 (a \sqrt{p}+1) (a^3 \sqrt{p}-1) (a^2+a \sqrt{p}+p) (a^2 p+a \sqrt{p}+1) (a-p^{3/2}) (a p^{3/2}-1) (a^2-p^2) (a^2 p^2-1))^{-1} $ \\
	\hline
	$a_8$ & $-a^2 p^{12}((a^2-1) (p-1)^3 (p+1) (a-\sqrt{p})^3 (a^3-\sqrt{p}) (a+\sqrt{p}) (a \sqrt{p}-1)^3 (a \sqrt{p}+1) (a-p^{3/2}) (a p^{3/2}-1) (a^2-p^2) (a-p^{5/2}))^{-1} $ \\
	\hline
	$a_9$ & $-a^{17} p^{12}((a^2-1) (p-1)^3 (p+1) (a-\sqrt{p})^3 (a+\sqrt{p}) (a \sqrt{p}-1)^3 (a^4 p+a^3 \sqrt{p}-a \sqrt{p}-1) (a-p^{3/2}) (a p^{3/2}-1) (a^3 p^{9/2}-a^2 p^2-a p^{5/2}+1))^{-1} $ \\ 
	\hline
	$a_{10}$ & $-a^6 p^{18}((p-1)^5 (p+1) (p^2+p+1) (a-\sqrt{p})^3 (a+\sqrt{p}) (a \sqrt{p}-1)^3 (a \sqrt{p}+1) (a-p^{3/2}) (a p^{3/2}-1) (a-p^{5/2}) (a p^{5/2}-1))^{-1} $ \\
	\hline
	$a_{11}$ & $a^2 p^{49/2}((a^2-1) (p-1)^3 (p+1) (p^2+p+1) (a \sqrt{p}-1) (a^2-a p^{3/2}-a \sqrt{p}+p^2)^3  (a^4+a^3 p^{3/2}+2 a^3 \sqrt{p}+2 a^2 p^2+2 a^2 p+a p^{3/2}+2 a p^{5/2}+p^3)(a^5 p^{3/2}-a^4 p^5-a^4 p^4-a^4+a^3 p^{5/2}+a^3 p^{15/2}+a^2 p^7+a^2 p^2-a p^{9/2}-a p^{11/2}-a p^{19/2}+p^8)^{-1} $ \\
	\hline
	$a_{12}$ & $a^{17} p^{49/2}((a^2-1) (p-1)^3 (p+1) (p^2+p+1) (a-\sqrt{p}) (a \sqrt{p}-1)^3 (a \sqrt{p}+1) (a^2 p+a \sqrt{p}+1)(a-p^{3/2})(a p^{3/2}-1)^3 (a p^{3/2}+1) (a^2 p^2-1)(a p^{5/2}-1) (a p^{7/2}-1)^{-1} $ \\
	\hline
	$a_{13}$ & $a^6 p^{32}((p-1)^5 (p+1)^3 (p^4+p^3+2 p^2+p+1) ( a^7 p^{3/2}+a^5 p^{3/2}-a^9 p^{5/2}+a^7 p^{5/2}+a^5 p^{5/2}-a^3 p^{5/2}-a^9 p^{7/2}+2 a^7 p^{7/2}+2 a^5 p^{7/2}-a^3 p^{7/2}  -a^9 p^{9/2}  +3 a^7 p^{9/2}+3 a^5 p^{9/2}-a^3 p^{9/2}-2 a^9 p^{11/2}+3 a^7 p^{11/2}+3 a^5 p^{11/2}-2 a^3 p^{11/2}-a^{11} p^{13/2}-3 a^9 p^{13/2}+3 a^7 p^{13/2}+3 a^5 p^{13/2}-3 a^3 p^{13/2}-a p^{13/2}-a^{11} p^{15/2}-4 a^9 p^{15/2}+3 a^7 p^{15/2}+3 a^5 p^{15/2}-4 a^3 p^{15/2}  -a p^{15/2}-a^{11} p^{17/2}-4 a^9 p^{17/2}+3 a^7 p^{17/2}+3 a^5 p^{17/2}-4 a^3 p^{17/2}-a p^{17/2}-a^{11} p^{19/2}-4 a^9 p^{19/2}+4 a^7 p^{19/2}+4 a^5 p^{19/2}-4 a^3 p^{19/2}-a p^{19/2}-a^{11} p^{21/2}-4 a^9 p^{21/2}+4 a^7 p^{21/2}+4 a^5 p^{21/2}-4 a^3 p^{21/2}  -a p^{21/2}-a^{11} p^{23/2}-4 a^9 p^{23/2}+3 a^7 p^{23/2}+3 a^5 p^{23/2}-4 a^3 p^{23/2}-a p^{23/2}-a^{11} p^{25/2}-4 a^9 p^{25/2}+3 a^7 p^{25/2}+3 a^5 p^{25/2}-4 a^3 p^{25/2}-a p^{25/2}-a^{11} p^{27/2}-3 a^9 p^{27/2}+3 a^7 p^{27/2}+3 a^5 p^{27/2}-3 a^3 p^{27/2}  -a p^{27/2}-2 a^9 p^{29/2}+3 a^7 p^{29/2}+3 a^5 p^{29/2}-2 a^3 p^{29/2}-a^9 p^{31/2}+3 a^7 p^{31/2}+3 a^5 p^{31/2}-a^3 p^{31/2}-a^9 p^{33/2}+2 a^7 p^{33/2}+2 a^5 p^{33/2}-a^3 p^{33/2}-a^9 p^{35/2}+a^7 p^{35/2}+a^5 p^{35/2}-a^3 p^{35/2}+a^7 p^{37/2}+a^5 p^{37/2}+a^7 p^{39/2}+a^5 p^{39/2}-a^6 p^{20}-a^6 p^{19}-2 a^6 p^{18}-3 a^6 p^{17}+a^{10} p^{16}-4 a^6 p^{16}+a^2 p^{16}+a^{10} p^{15}+a^8 p^{15}-4 a^6 p^{15}+a^4 p^{15}+a^2 p^{15}+2 a^{10} p^{14}+a^8 p^{14}-5 a^6 p^{14}+a^4 p^{14}+2 a^2 p^{14}+2 a^{10} p^{13}+a^8 p^{13}-6 a^6 p^{13}+a^4 p^{13}+2 a^2 p^{13}+2 a^{10} p^{12}+a^8 p^{12}-7 a^6 p^{12}+a^4 p^{12}+2 a^2 p^{12}+3 a^{10} p^{11}+2 a^8 p^{11}-6 a^6 p^{11}+2 a^4 p^{11}+3 a^2 p^{11}+a^{12} p^{10}+4 a^{10} p^{10}+3 a^8 p^{10}-6 a^6 p^{10}+3 a^4 p^{10}+4 a^2 p^{10}+p^{10}+3 a^{10} p^9+2 a^8 p^9-6 a^6 p^9+2 a^4 p^9+3 a^2 p^9+2 a^{10} p^8+a^8 p^8-7 a^6 p^8+a^4 p^8+2 a^2 p^8+2 a^{10} p^7+a^8 p^7-6 a^6 p^7+a^4 p^7+2 a^2 p^7+2 a^{10} p^6+a^8 p^6-5 a^6 p^6+a^4 p^6+2 a^2 p^6+a^{10} p^5+a^8 p^5-4 a^6 p^5+a^4 p^5+a^2 p^5+a^{10} p^4-4 a^6 p^4+a^2 p^4-3 a^6 p^3-2 a^6 p^2-a^6 p-a^6+a^7 \sqrt{p}+a^5 \sqrt{p})^{-1}$ \\ 
	\hline
	$a_{14}$ & $a^2 p^7((a^2-1)^3 (a^2+1) (a-\sqrt{p})^3 (a^3-\sqrt{p}) (a+\sqrt{p}) (a \sqrt{p}-1)^3 (a \sqrt{p}+1) (a^3 \sqrt{p}-1) (a^2+a \sqrt{p}+p) (a^2 p+a \sqrt{p}+1) (a-p^{3/2}) (a p^{3/2}-1) (a^2-p^2) (a^2 p^2-1))^{-1} $ \\ 
	\hline
	$a_{15}$ & $a^6 p^3((p-1)^5 (p+1) (p^2+p+1) (a-\sqrt{p})^3 (a+\sqrt{p}) (a \sqrt{p}-1)^3 (a \sqrt{p}+1)(a-p^{3/2}) (a p^{3/2}-1) (a-p^{5/2}) (a p^{5/2}-1))^{-1} $ \\ 
	\hline
	$a_{16}$ & $-a^{17} p^2((a^2-1) (p-1)^3 (p+1) (p^2+p+1) (a-\sqrt{p})^3 (a+\sqrt{p}) (a \sqrt{p}-1) (a^2+a \sqrt{p}+p) (a-p^{3/2})^3 (a+p^{3/2}) (a p^{3/2}-1) (a^2-p^2) (a-p^{5/2}) (a-p^{7/2}))^{-1} $ \\	
	\hline
\end{tabular}
\end{center}

\end{document}